\documentclass[a4paper]{amsart}
\usepackage{url,graphicx, amsmath, amsthm, amssymb, color, hyperref, xurl, tikz, color,soul}
\usepackage[foot]{amsaddr}
\usepackage{verbatim}
\newtheorem{theorem}{Theorem}[section]

\newtheorem{lemma}[theorem]{Lemma}
\newtheorem{corollary}[theorem]{Corollary}

\newtheorem{definition}{Definition}[section]

\title{On smooth elements in arithmetic semigroups}
\author{Giovanna De Lauri$^1$, Dan Fretwell$^2$, Jack Ye$^3$}
\address{$^1$ Lancaster University, \url{g.delauri@lancaster.ac.uk}}
\address{$^2$ Lancaster University, \url{d.fretwell@lancaster.ac.uk}} \address{$^3$ Lancaster University, \url{h.ye7@lancaster.ac.uk}}
\date{}

\begin{document}

\begin{abstract}
 In this paper we explore the concept of smoothness in arithmetic semigroups. In particular, we give explicit asymptotics for the counting function $\Psi_G(x,y)$ when the arithmetic semigroup $G$ satisfies Axiom A (in the spirit of Knopfmacher \cite{knopfmacher2015abstract}). The asymptotics generalise the standard results in the literature for the integer case, and are expressed in terms of the Dickman $\rho$ function. Finally, we give applications of our results to counting smooth families of finite abelian groups and finite semisimple rings.
\end{abstract}

\maketitle

\section{Introduction}

A positive integer
$n$ is called \emph{$y$-smooth} if every prime
factor of $n$ is at most $y$. Smooth numbers play an important role in many computational problems, for example in integer factorisation algorithms such as the Quadratic Sieve, where it is desirable to work with integers only having ``small" prime factors for  the Linear Algebra phase of the algorithm to work well (choosing a smoothness bound allows us to usefully model factorisations as finite vectors of exponents). See Pomerance's work \cite{pomerance1996tale, pomerance2008smooth} for a gentle introduction to this algorithm. \\

In order to analyse the performance of the Quadratic Sieve, it is necessary to have asymptotics for the counting function of $y$-smooth numbers:
$$
  \Psi(x,y) := \#\{n\le x : n \text{ is } y\text{-smooth}\}.
$$
Much is known about this function. In the 1930's, Dickman \cite{dickman1930frequency} introduced the function $\rho(u)$ satisfying $\rho(u)=1$ for $0\le u\le1$ and $u\rho'(u)+\rho(u-1)=0$ for $u>1$. Writing $x = y^u$ and letting $x \to \infty$, the following asymptotic is then known to hold: 
$$\Psi(x,y) \sim x \rho(u). $$
The error term in this asymptotic was provided in the 1950's by de Bruijn \cite{de1951number}. He was able to show that uniformly for a wide range of $u$: 
$$
\Psi(x,y) = x\rho(u)\left(1+O_\varepsilon\left(\frac{\log(u+1)}{\log y}\right)\right)
$$
with $x \geq 2$ and a specific bound on $y$, for any fixed $\varepsilon$. Hildebrand
\cite{hildebrand1986number} later proved that the same estimate holds uniformly for $
  1 \le u \le \exp\big((\log y)^{3/5-\varepsilon}\big)$ which remains the best result. A comprehensive account of this can be found in the survey of Hildebrand
and Tenenbaum \cite{hildebrand1993integers}. In \cite{weingartner2026explicit}, Weingartner gives fully explicit upper and lower bounds for the Dickman $\rho(u)$ function. \\

The classical proofs of the above results utilize an identity of Buchstab, which roughly speaking gives that the density of smooth numbers in a given interval depends recursively on the density at preceding intervals. \\

Generalisations of the notion of smoothness exist. For example if $K$ is a number field with ring of integers $\mathcal{O}_K$, then we can use the norm as a notion of size and consider an element $\alpha\in\mathcal{O}_K$ to be $y$-smooth if its principal ideal $\langle\alpha\rangle$ has norm that is $y$-smooth (i.e.\ the ideal is only divisible by prime ideals of small norm). Understanding the asymptotics of the associated counting function in this case is vital when analysing the performance of the Number Field Sieve \cite{buhler2008algorithmic}, the wiser big brother of the Quadratic Sieve. The analysis of such sieves is important when benchmarking the cryptographic security of schemes depending on integer factorisation and discrete log problems, e.g.\ RSA and Diffie-Hellman key exchange (see \cite{buhler2008algorithmic}).\\

In this paper, we generalise the notion of smoothness to elements of arithmetic semigroups (in the spirit of Knopfmacher
\cite{knopfmacher2015abstract}). Roughly speaking, an arithmetic semigroup is a
commutative semigroup $G$ where elements possess a unique prime factorisation into finitely many primes (for some countable subset of primes) and have a compatible norm function satisfying many natural properties (precise definitions will be given later). Examples include the set of natural numbers with the usual norm, and more generally the set of ideals of $\mathcal{O}_K$ with the ideal norm. However, many more examples exist, e.g.\ finite abelian groups and finite semisimple rings (see \cite{knopfmacher2015abstract} for more discussion).\\

One of the main goals of the book \cite{knopfmacher2015abstract} is to generalise the Prime Number Theorem to arithmetic semigroups. This is achieved under the assumption that Axiom A holds for the arithmetic semigroup, i.e. the mild assumption that the number of elements of norm bounded by $x$ grows like $Ax^{\delta}$ for some $\delta \geq 0$. Arithmetic semigroups satisfying Axiom A have a much more controllable arithmetic structure than others, and as such are the best candidates to consider when attempting to generalise classical results over the integers. For example, under this assumption we are able to establish the following smoothness asymptotics that are analogous to those over the integers and rings of integers of number fields.

\begin{theorem}
Let $G$ be an arithmetic semigroup satisfying Axiom A. For every fixed $U>0$, as $x\to\infty$,
\[
\Psi_G(x,y)
=
Ax^\delta\rho(u)
+O\!\left(\frac{x^\delta}{\log x}\right),
\qquad
u=\frac{\log x}{\log y},
\]
uniformly for $y\ge2$ and $0<u\le U$.
\end{theorem}

Prior to this work, smoothness in arithmetic \emph{graded} semigroups (arithmetic semigroups with an integer-valued
completely additive degree function) had already been studied (e.g.\ in work of Warlimont \cite{warlimont1991arithmetical} and Manstavi\v{c}ius \cite{manstavivcius1992remarks}). In the graded setting, the norm is determined by an integer-valued degree $|g|=q^{\deg g}$ for a fixed base $q>1$. Smoothness can therefore be expressed as a bound on the degrees of the prime factors. Our setting instead focuses on general real-valued multiplicative norms, and we count elements cumulatively up to a norm bound. Thus the norms are not required to be confined to powers of a single fixed base. Hence the results of Warlimont and Manstavi\v{c}ius do not directly apply to the applications in Section \ref{sec:app}, justifying the importance of Theorem~\ref{thm:main}.\\

We illustrate the result via applications to finite abelian groups and finite semisimple rings, counted up to isomorphism. In these settings, smoothness bounds the sizes of the prime power factors and simple factors respectively. Our theorem gives quantitative asymptotics for the number of such objects of bounded size.\\

The paper is structured in the following way, Section \ref{sec:setup} introduces the setting and the auxiliary results, Section \ref{sec:main} contains the main theorem and its proof and Section \ref{sec:app} follows with the applications.

\section{Acknowledgments}
This work was carried out during the PhD of the first and third named authors, and forms part of the thesis of the third named author. We would like to thank Daniel Shiu for helpful insights into further work. Additionally, we are grateful for kind advice given to us by Lian Duan, Ning Ma, and Shaoyun Yi \cite{duan2022generalizations}, who have worked on the generalisation of Alladi's formula to
arithmetic semigroups under Axiom A and Axiom A$^\#$.

 \section{Setup and Notation}\label{sec:setup}
In this section, we introduce the framework and notation for studying smooth elements in arithmetic semigroups satisfying Axiom A. We first recall the relevant background and collect the results required for our main theorem.
\subsection{Arithmetic semigroups}

Most of the definitions and results in this subsection are taken from Knopfmacher \cite{knopfmacher2015abstract}.

\begin{definition}\label{def:arithmetic-semigroup}

Let $G$ be a commutative semigroup with identity element $1\in G$. Suppose that $P\subseteq G$ is a finite or countably infinite subset such that every $g\in G\backslash\{1\}$ has a finite factorisation of the form: \[g = p_1^{a_1}p_2^{a_2}...p_n^{a_n},\] that is unique up to ordering, with $n\geq 1$, each $p_i\in P$ distinct and each $a_i\geq 1$.

        We call $G$ an arithmetic semigroup if in addition there exists a real-valued \textit{norm mapping}
\[
|\cdot|:G\longrightarrow[1,\infty)
\]
satisfying:

\begin{itemize}
    \item[(i)] $|1| = 1$, $|p| > 1$ for $p \in P$,
    \item[(ii)] $|fg| = |f| |g|$ for all $f,g \in G$,
    \item[(iii)] the counting function $N_G(x) =\#\{g\in G:|g|\le x\}$ is finite, for each real $x > 0$.
\end{itemize}
\end{definition}

The above conditions imply that the prime counting function \[\pi_G(x) = \#\{g\in P\,|\, |g|\le x\}\] is finite for each real $x>0$.\\

\noindent \textbf{Examples.}
\textit{
Below are a couple of examples of arithmetic semigroups. More can be found in \cite[Chapter 1]{knopfmacher2015abstract}.
\begin{itemize}
\item{The most classical example is given by $G = \mathbb{N}_{\geq 1}$, with $P \subseteq G$ being set of prime numbers and the norm mapping being the standard norm $|n| = n$. We have $N_G(x) = \lfloor x \rfloor\sim x$.}
\item{Another simple example arises from the Gaussian integers. Let $G = \mathbb{Z}[i]/\{\pm 1,\pm i\}$ (i.e.\ associate classes of elements) and let $P \subseteq G$ be the classes represented by Gaussian primes (i.e.\ $1+i$, primes $p\equiv 3 \bmod 4$ and $a+bi$ where $a^2+b^2=p$ and $p\equiv 1 \bmod 4$). The norm mapping is the field norm $|a+bi| = a^2+b^2$ and then $N_G(x) = \frac{1}{4} \sum_{n \leq x}r(n)$ with $r(n) = 
    \#\{(a,b) \in \mathbb{R}^2: a^2+b^2 = n\}$, so that $N_G(x)\sim \frac{\pi}{4}x$.}
\end{itemize}}

 We can define the notion of smoothness for elements of arithmetic semigroups in the obvious way.

 \begin{definition} \label{def:smooth}
Let $G$ be an arithmetic semigroup and let $y\geq 1$. We say that $g\in G$ is $y$-smooth if every prime $p_i\in P$ appearing in the factorisation of $g$ satisfies $|p_i| \leq y$.

\noindent For any $x\geq 1$  and $1 \leq y \leq x$ we let \[S_G(x,y) = \{g\in G\,|\, |g|\leq x, |p_i|\leq y \text{ for all } i\}\] be the set of $y$-smooth elements of $G$ of norm $x$ or less and we let $\Psi_G(s,y) = \#S_G(x,y)$ count these elements.
\end{definition}

Note that $\Psi_G(x,y)$ is necessarily finite and that $1\in S_G(x,y)$. For $0<x<1$, we may set
$S_G(x,y)=\varnothing$ and $\Psi_G(x,y)=0$. Thus
\[
0\le\Psi_G(x,y)\le N_G(x)
\qquad(x,y\ge1).
\] 
We wish to understand asymptotics for the function $\Psi_G(x,y)$.

       To study $\Psi_G(x,y)$, we require estimates for the
distribution of prime elements in $G$. We therefore require asymptotics for the prime-counting function $\pi_G(x)$.\\

        A key result in the book by Knopfmacher \cite{knopfmacher2015abstract} is the following, which we will assume throughout Sections \ref{sec:setup} and \ref{sec:main}. Suppose that $G$ is an arithmetic semigroup satisfying ``Axiom A", i.e.\ that there exist positive constants $A,\delta$ and $0\leq \eta < \delta$ such that 
        \begin{equation}\label{axiom}
        N_G(x) = Ax^{\delta} + O(x^{\eta})\qquad(x\to\infty).
        \end{equation}
        It is proven in \cite[Chapter VI]{knopfmacher2015abstract} for any arithmetic semigroup $G$ satisfying Axiom A that
        \begin{equation}\label{primecounting}
    \pi_G(x)\sim\frac{x^\delta}{\delta\log x}\quad(x\rightarrow\infty).
        \end{equation}
        Quite a few classical examples of Prime Number Theorems are special cases of this result (e.g. the classical Prime Number Theorem and Landau's Prime Number Theorem \cite{landau2010einfuhrung} for ideals of bounded norm in rings of integers of number fields).

\subsection{The Dickman \texorpdfstring{$\rho$}{rho} function}
A lot is known about smoothness in the case where $G = \mathbb{N}_{\geq 1}$ under multiplication, $P$ is the set of primes and $|\cdot|$ is the usual absolute value.
The Dickman $\rho$ function describes the limiting proportion of smooth integers when $u=\log x/\log y$ is fixed. The same function appears in our asymptotic formula for $\Psi_G(x,y)$ under Axiom A. We record its definition and the properties needed in the proof of our main theorem.
\begin{definition}\label{def:dickman}
The Dickman function is the unique continuous function
$\rho:[0,\infty)\to\mathbb R$, differentiable on
$(1,\infty)$, satisfying
\[
\rho(u)=1\qquad(0\le u\le1)
\]
and
\[
u\rho'(u)+\rho(u-1)=0\qquad(u>1).
\]
\end{definition}

Integrating the defining differential equation gives,
for every integer $N\ge1$,
\begin{equation}\label{rho-recursion}
\rho(u)
=
\rho(N)-\int_N^u\frac{\rho(v-1)}{v}\,dv,
\qquad N\le u\le N+1.
\end{equation}
This identity determines $\rho$ successively on the
intervals $[N,N+1]$. In particular,
\[
\rho(u)=1-\log u\qquad(1\le u\le2).
\]

The function $\rho$ is strictly positive on $[0,\infty)$,
equals $1$ on $[0,1]$, and is strictly decreasing on
$(1,\infty)$. Hence, for every fixed $U>0$,
\[0<\rho(U)\le\rho(u)\le1
\qquad(0\le u\le U).\]
For the definition and proofs of these properties, see
Hildebrand-Tenenbaum
\cite[equations (1.5)-(1.6), p. 414,
and Lemma 2.5, p. 426]{hildebrand1993integers}.

\subsection{Auxiliary estimates}

We establish the estimates needed to prove our main theorem
uniformly for $0<u\le U$, for every fixed $U>0$.
The proof adapts the classical induction based on the
largest-prime-factor decomposition; see
\cite[Section 1, p. 416]{hildebrand1993integers}.

\begin{lemma}[Generalised Mertens' estimate]\label{lem:mertens}
There is a constant $B_G$ such that
\[
T(x):=\sum_{\substack{p\in P\\|p|\le x}}|p|^{-\delta}
=
\log\log x+B_G
+O\!\left(\frac1{\log x}\right)
\qquad(x\ge2).
\]
\end{lemma}

\begin{proof}
See Knopfmacher
\cite[Chapter VI, Lemma 2.5]{knopfmacher2015abstract}.
\end{proof}

\begin{lemma}[An integral estimate]\label{lem:integral}
For fixed $\alpha>0$ and $c>1$, as $x\to\infty$,
\[
\int_c^x\frac{t^{\alpha-1}}{\log t}\,dt
=
\frac{x^\alpha}{\alpha\log x}
+
O_{\alpha,c}\!\left(
\frac{x^\alpha}{(\log x)^2}
\right).
\]
\end{lemma}

\begin{proof}
Integration by parts gives
\[
\int_c^x\frac{t^{\alpha-1}}{\log t}\,dt
=
\frac{x^\alpha}{\alpha\log x}
-\frac{c^\alpha}{\alpha\log c}
+\frac1\alpha
\int_c^x\frac{t^{\alpha-1}}{(\log t)^2}\,dt.
\]
For $x\ge c^2$, we split the remaining integral at $\sqrt{x}$.
On $[c,\sqrt{x}]$,
\[
\int_c^{\sqrt{x}}
\frac{t^{\alpha-1}}{(\log t)^2}\,dt
\le
\frac1{(\log c)^2}
\int_c^{\sqrt{x}}t^{\alpha-1}\,dt
\ll_{\alpha,c}x^{\alpha/2}.
\]
On $[\sqrt{x},x]$,
\[
\int_{\sqrt{x}}^x
\frac{t^{\alpha-1}}{(\log t)^2}\,dt
\le
\frac4{(\log x)^2}
\int_{\sqrt{x}}^x t^{\alpha-1}\,dt
\ll_\alpha\frac{x^\alpha}{(\log x)^2}.
\]
Since
\[
x^{\alpha/2}
=o\!\left(\frac{x^\alpha}{(\log x)^2}\right),
\]
the result follows.
\end{proof}

\begin{lemma}[Weighted prime sums]\label{lem:weighted}
For the exponent $0\le\eta<\delta$ in Axiom A,
\[
\sum_{\substack{p\in P\\|p|\le x}}|p|^{-\eta}
\ll\frac{x^{\delta-\eta}}{\log x}
\qquad(x\ge2).
\]
\end{lemma}

\begin{proof}
Recall from Lemma \ref{lem:mertens} that
\[
T(t):=\sum_{\substack{p\in P\\|p|\le t}}|p|^{-\delta}
=\log\log t+B_G+O(1/\log t).
\]
Choose $q_0>1$ below every prime norm, which is
possible by local finiteness and since every prime has norm $|p| > 1$. Then $T(q_0)=0$.
By enlarging the implied constant on a fixed initial
interval, Mertens' estimate holds for all $t\ge q_0$.

Since $|p|^{-\eta}=|p|^{-\delta}|p|^{\delta-\eta}$,
partial summation with cumulative sum $T(t)$ and
weight $t^{\delta-\eta}$ gives, for $x\ge q_0$, then
\begin{equation}\label{abel}
\sum_{\substack{p\in P\\|p|\le x}}|p|^{-\eta}
=x^{\delta-\eta}T(x)-(\delta-\eta)\int_{q_0}^x
T(t)t^{\delta-\eta-1}\,dt.
\end{equation}
The lower-endpoint term vanishes because $T(q_0)=0$.

Substituting the Mertens' estimate into \eqref{abel}
and integrating the main terms by parts, we obtain
\begin{align*}
\sum_{\substack{p\in P\\|p|\le x}}|p|^{-\eta}
&=
q_0^{\delta-\eta}(\log\log q_0+B_G)
+\int_{q_0}^x\frac{t^{\delta-\eta-1}}{\log t}\,dt\\
&\quad+
O\!\left(
\frac{x^{\delta-\eta}}{\log x}
+\int_{q_0}^x\frac{t^{\delta-\eta-1}}{\log t}\,dt
\right).
\end{align*}
Here the terms $x^{\delta-\eta}\log\log x$ and
$B_Gx^{\delta-\eta}$ have cancelled exactly.
Since $\delta-\eta>0$, Lemma \ref{lem:integral} gives
\[
\sum_{\substack{p\in P\\|p|\le x}}|p|^{-\eta}
\ll\frac{x^{\delta-\eta}}{\log x}
\]
for sufficiently large $x$.
\end{proof}
In an arithmetic semigroup, distinct primes can have the same
norm. We use the inclusion-exclusion identity below to account for
this possibility.
For each $t$ that is the norm of a prime element, let
\[
d_t=\#\{p\in P:|p|=t\}.
\]
Note that all sums indexed by $t$ below run over the possible distinct norms
of prime elements.

\begin{lemma}[Largest-prime-norm decomposition]\label{lem:ie}
For $x,y\ge1$,
\begin{equation}\label{ie}
\Psi_G(x,y)
=
1+\sum_{t\le y}\sum_{k=1}^{d_t}
(-1)^{k-1}\binom{d_t}{k}\Psi_G(x/t^k,t).
\end{equation}
\end{lemma}

\begin{proof}
For each $t$ that is the norm of a prime element, define
\[
P_t=\{p\in P:|p|=t\},
\quad \text{and} \quad 
A_p(t)=\{a\in S_G(x,t):p\mid a\}
\quad \text{for}\,\, p\in P_t.
\]
These sets are finite and $\#P_t=d_t$. We have
\[
\bigcup_{p\in P_t}A_p(t)
=
\left\{
a\in G\setminus\{1\}:|a|\le x,\ 
\max_{\substack{p\in P\\p\mid a}}|p|=t
\right\}.
\]
Indeed, membership in $S_G(x,t)$ bounds every prime factor
norm by $t$, while divisibility by some $p\in P_t$ ensures
that the bound is reached.

Let $J\subseteq P_t$ be non-empty, where $\#J=k$.
Since the primes in $J$ are distinct, unique factorisation
gives
\[
\bigcap_{p\in J}A_p(t)
=
\left\{
a\in S_G(x,t):
\prod_{p\in J}p\mid a
\right\}.
\]
Consider the map
\[
S_G(x/t^k,t)\longrightarrow\bigcap_{p\in J}A_p(t),
\qquad
b\longmapsto b\prod_{p\in J}p.
\]
Since
\[
\left|\prod_{p\in J}p\right|=t^k,
\]
this map is well-defined. Conversely, every element $a$
of the intersection can be written uniquely as
\[
a=b\prod_{p\in J}p.
\]
The quotient satisfies
\[
|b|=\frac{|a|}{t^k}\le\frac{x}{t^k},
\]
and every prime divisor of $b$ has norm at most $t$.
Thus $b\in S_G(x/t^k,t)$, proving surjectivity and
uniqueness of the preimage. The map is therefore a
bijection. If $x/t^k<1$ (as, by convention, $\Psi_G(z, t) = 0$ for $z < 1$), both sets are empty.

Consequently,
\[
\#\bigcap_{p\in J}A_p(t)=\Psi_G(x/t^k,t).
\]
Finite inclusion-exclusion gives
\begin{align*}
\#\bigcup_{p\in P_t}A_p(t)
&=
\sum_{\varnothing\ne J\subseteq P_t}
(-1)^{\#J-1}
\#\bigcap_{p\in J}A_p(t)\\
&=
\sum_{k=1}^{d_t}
(-1)^{k-1}\binom{d_t}{k}\Psi_G(x/t^k,t),
\end{align*}
because $P_t$ has exactly $\binom{d_t}{k}$ subsets
of cardinality $k$.

Finally, every non-identity element has a uniquely
determined largest prime factor norm. Hence
\[
S_G(x,y)
=
\{1\}\sqcup
\bigsqcup_{t\le y}
\left(\bigcup_{p\in P_t}A_p(t)\right).
\]
By taking cardinalities and substituting the
inclusion-exclusion formula we prove \eqref{ie}.
\end{proof}

Note that when different prime elements have different norms, i.e.
$d_t=1$ for each $t$, then \eqref{ie} reduces to
\[
\Psi_G(x,y)
=
1+\sum_{\substack{p\in P\\|p|\le y}}
\Psi_G(x/|p|,|p|).
\]
For $G=\mathbb N_{\geq 1}$, this is the classical Buchstab identity
used in the classical smooth-number induction proof in \cite[Chapter 9, Theorem 9.3]{de2023analytic}. Formula \eqref{ie}
is its replacement when distinct primes can have equal norms.\\
To use \eqref{ie} in the induction, we must control
the correction terms with $k\ge2$, which arise when
distinct prime elements have the same norm.
The following estimate will bound their total
contribution over large prime norms.
\begin{lemma}[Equal-norm prime pairs]\label{lem:collision}
For $L\ge2$, we have that
\[
\sum_{t>L}d_t^2t^{-2\delta}
\ll\frac1{\log L}.\]
\end{lemma}

\begin{proof}
For $T\ge2$, formula \eqref{primecounting} gives
\[
\sum_{T<t\le2T}d_t
=
\pi_G(2T)-\pi_G(T) \leq \pi_G(2T) \ll \frac{{(2T)}^\delta}{\log 2T}
\ll\frac{T^\delta}{\log T}.
\]
Therefore
\begin{align*}
\sum_{T<t\le2T}d_t^2t^{-2\delta}
&\le
T^{-2\delta}\sum_{T<t\le2T}d_t^2\\
&\le
T^{-2\delta}
\left(\sum_{T<t\le2T}d_t\right)^2\\
&\ll\frac1{(\log T)^2}.
\end{align*}
Summing over the intervals $(2^jL,2^{j+1}L]$, we obtain
\[
\sum_{t>L}d_t^2t^{-2\delta}
\ll
\sum_{j=0}^{\infty}
\frac1{(\log L+j\log2)^2}.
\]
Since the summand is decreasing as a function of $j$,
the last sum is at most
\begin{align*}
\frac1{(\log L)^2}
+\int_0^\infty
\frac{ds}{(\log L+s\log2)^2}
&=
\frac1{(\log L)^2}
+\frac1{\log2\,\log L}\\
&\ll\frac1{\log L},
\end{align*}
since $\log L \geq \log 2$. This proves the result.
\end{proof}

\begin{lemma}[Stieltjes summation with the Dickman weight]
\label{lem:stieltjes}
For every fixed integer $N\ge2$, uniformly for
$N<\tilde u\le N+1$ as $x\to\infty$, we have that
\begin{align*}
&\sum_{\substack{p\in P\\
x^{1/\tilde u}<|p|\le x^{1/N}}}
|p|^{-\delta}
\rho\!\left(\frac{\log x}{\log|p|}-1\right)=
\int_N^{\tilde u}\frac{\rho(v-1)}v\,dv
+O\!\left(\frac1{\log x}\right).
\end{align*}
\end{lemma}

\begin{proof}
By the definition of $T$, the sum equals
\[
\int_{(x^{1/\tilde u},\,x^{1/N}]}
\rho\!\left(\frac{\log x}{\log t}-1\right)\,dT(t).
\]
The half-open interval agrees with the strict lower
endpoint in the prime sum.

By Lemma \ref{lem:mertens},
\[
\sup_{x^{1/\tilde u}\le t\le x^{1/N}}
|T(t)-\log\log t-B_G|
\ll\frac{\tilde u}{\log x}
\ll\frac1{\log x}.
\]
On this interval, the function
\[
t\longmapsto
\rho\!\left(\frac{\log x}{\log t}-1\right)
\]
is continuous and nondecreasing, taking values in $[0,1]$.
Also, $T(t)-\log\log t-B_G$ has bounded variation on this
compact interval.

Stieltjes integration by parts therefore gives
\begin{align*}
&\left|
\int_{(x^{1/\tilde u},\,x^{1/N}]}
\rho\!\left(\frac{\log x}{\log t}-1\right)
\,d\bigl(T(t)-\log\log t-B_G\bigr)
\right|\\
&\qquad\le
3\sup_{x^{1/\tilde u}\le t\le x^{1/N}}
|T(t)-\log\log t-B_G|\\
&\qquad\ll\frac1{\log x}.
\end{align*}
It follows that
\begin{align*}
&\sum_{\substack{p\in P\\
x^{1/\tilde u}<|p|\le x^{1/N}}}
|p|^{-\delta}
\rho\!\left(\frac{\log x}{\log|p|}-1\right)\\
&\qquad=
\int_{x^{1/\tilde u}}^{x^{1/N}}
\rho\!\left(\frac{\log x}{\log t}-1\right)
\frac{dt}{t\log t}
+O\!\left(\frac1{\log x}\right).
\end{align*}
With
\[
v=\frac{\log x}{\log t},
\qquad
\frac{dt}{t\log t}=-\frac{dv}{v},
\]
the endpoints $x^{1/\tilde u}$ and $x^{1/N}$ correspond
to $\tilde u$ and $N$, respectively. Hence
\[
\int_{x^{1/\tilde u}}^{x^{1/N}}
\rho\!\left(\frac{\log x}{\log t}-1\right)
\frac{dt}{t\log t}
=
\int_N^{\tilde u}\frac{\rho(v-1)}v\,dv.
\]
This proves the lemma.
\end{proof}

\section{The main theorem}\label{sec:main}
We now establish the Dickman asymptotic for smooth
elements of an arithmetic semigroup satisfying Axiom A,
uniformly when $u$ lies in a fixed bounded interval.
The proof follows the classical induction on the
smoothness parameter, using the largest-prime-norm
decomposition \eqref{ie}. The additional terms arising
from distinct primes with equal norms are controlled
by Lemma \ref{lem:collision}.

\begin{theorem}\label{thm:main}
Let $G$ be an arithmetic semigroup satisfying Axiom A
as in \eqref{axiom}. For every fixed $U>0$, as $x\to\infty$,
\begin{equation}\label{main}
\Psi_G(x,y)
=
Ax^\delta\rho(u)
+O\!\left(\frac{x^\delta}{\log x}\right),
\qquad
u=\frac{\log x}{\log y},
\end{equation}
uniformly for $y\ge2$ and $0<u\le U$.
\end{theorem}

\begin{proof}
We first prove \eqref{main} uniformly for $0<u\le2$.
We then proceed by induction, extending the uniform
estimate from $0<u\le N$ to $0<u\le N+1$.

\smallskip
\noindent\textbf{The range $0<u\le1$.}
Here $y=x^{1/u}\ge x$. Since every prime divisor of an
element of norm at most $x$ has norm at most $x$,
\[
\Psi_G(x,x^{1/u})=N_G(x)=Ax^\delta+O(x^\eta).
\]
Moreover,
\[
\frac{x^\eta}{x^\delta/\log x}
=
\frac{\log x}{x^{\delta-\eta}}
\longrightarrow0,
\]
because $\delta-\eta>0$. Since $\rho(u)=1$ for
$0<u\le1$, this proves \eqref{main} uniformly in this range.

\smallskip
\noindent\textbf{The range $1<u\le2$.}
Put $y=x^{1/u}$, so that $y\ge\sqrt{x}$.
An element of norm at most $x$ cannot be divisible by two prime
factors whose norms exceed $y$, as
their product would have norm greater than $y^2\ge x$.

Every element counted by $N_G(x)-\Psi_G(x,y)$ therefore
has a unique prime factor $p$ with $|p|>y$.
For a fixed such prime, division by $p$ provides a
bijection with the set of elements of norm at most $x/|p|$. Thus,
\[
\Psi_G(x,y)
=
N_G(x)
-\sum_{\substack{p\in P\\y<|p|\le x}}N_G(x/|p|).
\]

Using Axiom A for $x/|p|\ge1$ one obtains
\begin{align*}
\Psi_G(x,y)
&=
Ax^\delta
-Ax^\delta
\sum_{\substack{p\in P\\y<|p|\le x}}|p|^{-\delta}+
O\!\left(
x^\eta+
x^\eta
\sum_{\substack{p\in P\\y<|p|\le x}}|p|^{-\eta}
\right).
\end{align*}
By Lemma \ref{lem:weighted},
\[
x^\eta
\sum_{\substack{p\in P\\y<|p|\le x}}|p|^{-\eta}
\ll\frac{x^\delta}{\log x}.
\]
Note that the term $x^\eta$ is absorbed into the same bound.

Furthermore, Lemma \ref{lem:mertens} gives
\begin{align*}
\sum_{\substack{p\in P\\y<|p|\le x}}|p|^{-\delta}
&=
\log\log x-\log\log y
+O\!\left(\frac1{\log x}+\frac1{\log y}\right)\\
&=
\log u+O(1/\log x),
\end{align*}
uniformly for $1<u\le2$, since $\log y=(\log x)/u$.
Consequently,
\[
\Psi_G(x,x^{1/u})
=
Ax^\delta(1-\log u)
+O(x^\delta/\log x).
\]
As $\rho(u)=1-\log u$ on $[1,2]$, this establishes
\eqref{main} uniformly for $0<u\le2$.

\smallskip
\noindent\textbf{The induction step.}
Let $N\ge2$ be an integer. Suppose that
\begin{equation}\label{ih}
\Psi_G(x,x^{1/u})
=
Ax^\delta\rho(u)
+O\!\left(\frac{x^\delta}{\log x}\right)
\end{equation}
holds uniformly for $0<u\le N$, whenever
$x^{1/u}\ge2$. We prove the same estimate uniformly for
$N<\tilde u\le N+1$.

Subtracting \eqref{ie} at the thresholds $x^{1/N}$
and $x^{1/\tilde u}$ gives
\begin{equation}\label{subtract}
\Psi_G(x,x^{1/\tilde u})
=
\Psi_G(x,x^{1/N})-S_1-\sum_{i\ge2}S_i,
\end{equation}
where
\[
S_i
=
\sum_{x^{1/\tilde u}<t\le x^{1/N}}
(-1)^{i-1}\binom{d_t}{i}\Psi_G(x/t^i,t),
\]
and $\binom{d_t}{i}=0$ when $i>d_t$.
These sums are finite.

\smallskip
\noindent\emph{The terms with $i\ge2$.}
Axiom A implies $N_G(t)\ll t^\delta$ for $t\ge1$.
Together with our convention for arguments below $1$,
this yields
\[
\Psi_G(x/t^i,t)\ll x^\delta t^{-i\delta}.
\]
Hence,
\[\sum_{i\ge2}|S_i|
\ll
x^\delta
\sum_{x^{1/\tilde u}<t\le x^{1/N}}
\sum_{i=2}^{d_t}\binom{d_t}{i}t^{-i\delta}.
\]
By \eqref{primecounting},
\[
d_tt^{-\delta}
\le\pi_G(t)t^{-\delta}
\ll\frac1{\log t}.
\]
Since
\[
t>x^{1/\tilde u}\ge x^{1/(N+1)},
\]
we have that $d_tt^{-\delta}\le1$ uniformly over the
summation range for sufficiently large $x$.

Using the fact that $\binom{d_t}{i}\le d_t^i/i!$, we obtain
\begin{align*}
\sum_{i=2}^{d_t}\binom{d_t}{i}t^{-i\delta}
&\le
\sum_{i=2}^\infty\frac{(d_tt^{-\delta})^i}{i!}\\
&\le
d_t^2t^{-2\delta}\sum_{i=2}^\infty\frac1{i!}\\
&\ll d_t^2t^{-2\delta}.
\end{align*}
Applying Lemma \ref{lem:collision} with
$L=x^{1/\tilde u}$ therefore gives
\begin{equation}\label{error}
\sum_{i\ge2}|S_i|
\ll
\frac{x^\delta}{\log(x^{1/\tilde u})}
=
\frac{\tilde u x^\delta}{\log x}
\ll\frac{x^\delta}{\log x}.
\end{equation}

\smallskip
\noindent\emph{The principal sum $S_1$.}
Regrouping the terms according to norms of prime elements gives
\begin{equation}\label{regroup}
\begin{split}
S_1
&=
\sum_{x^{1/\tilde u}<t\le x^{1/N}}
d_t\Psi_G(x/t,t)\\
&=
\sum_{\substack{p\in P\\
x^{1/\tilde u}<|p|\le x^{1/N}}}
\Psi_G(x/|p|,|p|).
\end{split}
\end{equation}
For every prime in this range, we have that
\[
N-1
\le
\frac{\log(x/|p|)}{\log|p|}
=
\frac{\log x}{\log|p|}-1
<
\tilde u-1
\le N.
\]
Moreover,
\[
x/|p|\ge x^{1-1/N}\longrightarrow\infty,
\qquad
\log(x/|p|)\ge(1-1/N)\log x,
\]
uniformly over the range. Also,
$|p|>x^{1/(N+1)}\ge2$ for sufficiently large $x$.
Thus every summand lies in the range of the uniform
induction hypothesis.

Applying \eqref{ih} with norm bound $x/|p|$ and
smoothness threshold $|p|$, we obtain
\begin{align*}
\Psi_G(x/|p|,|p|)
&=
Ax^\delta|p|^{-\delta}
\rho\!\left(\frac{\log x}{\log|p|}-1\right)\\
&\quad+
O\!\left(
\frac{x^\delta|p|^{-\delta}}{\log x}
\right).
\end{align*}
By Lemma \ref{lem:mertens},
\begin{align*}
\sum_{\substack{p\in P\\
x^{1/\tilde u}<|p|\le x^{1/N}}}|p|^{-\delta}
&=
\log(\tilde u/N)+O(1/\log x)\\
&=O(1).
\end{align*}
The sum of the errors is therefore
$O(x^\delta/\log x)$, and
\begin{equation}\label{sone}
\begin{split}
S_1
&=
Ax^\delta
\sum_{\substack{p\in P\\
x^{1/\tilde u}<|p|\le x^{1/N}}}
|p|^{-\delta}
\rho\!\left(\frac{\log x}{\log|p|}-1\right)\\
&\quad+O(x^\delta/\log x).
\end{split}
\end{equation}

\smallskip
\noindent\emph{Completion of the induction.}
Lemma \ref{lem:stieltjes} applied to \eqref{sone} gives
\[
S_1
=
Ax^\delta
\int_N^{\tilde u}\frac{\rho(v-1)}v\,dv
+O(x^\delta/\log x).
\]
Combining this with \eqref{subtract}, \eqref{error},
and the induction hypothesis at $u=N$, we obtain
\begin{align*}
\Psi_G(x,x^{1/\tilde u})
&=
Ax^\delta
\left(
\rho(N)-\int_N^{\tilde u}\frac{\rho(v-1)}v\,dv
\right)\\
&\quad+O(x^\delta/\log x).
\end{align*}
By \eqref{rho-recursion}, the expression in parentheses
is $\rho(\tilde u)$. Hence
\[
\Psi_G(x,x^{1/\tilde u})
=
Ax^\delta\rho(\tilde u)
+O(x^\delta/\log x),
\]
uniformly for $N<\tilde u\le N+1$.
This completes the induction.

For any fixed $U>0$, finitely many induction steps cover the range
$0<u\le U$. Their constants can be absorbed into a single
constant depending on $G$ and $U$, proving \eqref{main}.
\end{proof}

An immediate corollary is the following result, describing the proportion of $y$-smooth elements among all elements of norm at most $x$.
\begin{corollary}\label{cor:smooth-proportion}
Under Theorem \ref{thm:main}, for every fixed $U>0$,
$$
\frac{\Psi_{G}(x,y)}{N_G(x)}
=
\rho(u)+O\!\left(\frac{1}{\log x}\right),
\qquad
u=\frac{\log x}{\log y},
$$
uniformly for $y\ge2$ and $0<u\le U$, as $x\to\infty$.
\end{corollary}
\begin{proof}
    It follows from direct application of Theorem \ref{thm:main} and Axiom A.
\end{proof}

A natural direction for further work is to extend the range of uniformity to allow $u$ to grow with $x$. This requires explicit control of the dependence of the error estimates on $u$, together with an error sufficiently small relative to $\rho(u)$, which tends to zero as $u\to\infty$.
Another suggestion would be investigating secondary terms in the smooth-element asymptotic. Retaining more information in the prime-sum estimates and the largest-prime-norm decomposition may permit a more precise expansion, possibly under stronger hypotheses on the arithmetic semigroups. Such an expansion would clarify whether the error term $O(x^\delta/\log x)$ can be reduced after the appropriate secondary term has been included.

\section{Applications}\label{sec:app}
 Besides the usual arithmetic semigroups satisfying Axiom A, i.e.\ positive integers under multiplication and non-zero integral ideals in the ring of integers of a number field, there are many other interesting examples to consider. 
 
 Finite abelian groups and finite semisimple rings provide two such examples, and Theorem \ref{thm:main} can be applied to give quantitative asymptotics for the number of isomorphism classes of bounded cardinality satisfying the corresponding smoothness restrictions. We explore these now.\\

\textbf{Example 1: Finite abelian groups}\\
Let $G_{ab}$ be the set of isomorphism classes of finite abelian
  groups with direct sum as operation and norm mapping $|F| = \# F$. Since finite abelian groups have unique decompositions into direct sums of prime power cyclic groups (up to ordering), we can take our set of primes to be $P_{ab} = \{\mathbb{Z}/p^r\mathbb{Z} \,|\, p \text{ prime}, r\ge 1\}$. Note that the associated arithmetic semigroup is not additively graded as $|G\oplus H| = |G|\cdot |H| \neq |G| + |H|$ in general (so that the results that follow are not implied by those in \cite{warlimont1991arithmetical} and \cite{manstavivcius1992remarks}).\\
  
  Let $\zeta_{G_{ab}}(s)=\sum_{n\geq1}a(n)n^{-s}$ where $a(n)$ defines the number of isomorphism classes of abelian groups of order $n$. Following Example 2.5 in \cite{knopfmacher2015abstract}, we have the pseudo-convergent double Euler product 
  $$\zeta_{G_{ab}}(s) = \prod_{r\geq1}\prod_p (1-p^{-rs})^{-1}
  = \prod_{r\geq1}\zeta(rs).$$
  Then, the residue of $\zeta_{G_{ab}}$ at $s=1$ is therefore the value of the remaining product at $s=1$ given by Corollary 1.3 in \cite{knopfmacher1970arithmetical}, 
  $$ A_{G_{ab}} = \prod_{r=2}^{\infty}\zeta(r) = 2.29485\dots  $$
    This is Knopfmacher's Corollary 1.2 \cite{knopfmacher2015abstract}, attributed there to Erd\H{o}s-Szekeres \cite[\S1,pp. 97--101]{erdos1934abelian}:
    $$
  N_{G_{ab}}(x) = \sum_{n\leq x} a(n) = \Big(\prod_{r=2}^{\infty}\zeta(r)\Big)x + O(x^{1/2}).$$
  Therefore, with $u = \log{x}/\log{y}$ and with $\delta_{G_{ab}}=1$, $A_{G_{ab}}=\prod_{r\ge2}\zeta(r)$, $\eta_{G_{ab}}=1/2$, our Theorem \ref{thm:main} leads us to
  $$
   \Psi_{G_{ab}}(x,y) = (1+o(1))\Big(\prod_{r=2}^{\infty}\zeta(r)\Big)\,x\,\rho(u).
   $$
Here $\delta=1$ follows from the linear growth of
$N_{G_{ab}}(x)$ established by Erd\H{o}s and Szekeres
\cite[\S1, pp. 97--101]{erdos1934abelian}.
Analytically, this is reflected in the simple pole
of $\zeta_{G_{ab}}(s)$ at $s=1$, arising from the
factor $\zeta(s)$, since the remaining product
$\prod_{r\ge2}\zeta(rs)$ is analytic and nonzero there.\\

\textbf{Example 2: Semisimple
finite rings}

Let $G_{sar}$ be the set of isomorphism classes of finite semisimple associative rings with direct sum as operation and norm function $|R| = \# R$. The Wedderburn-Artin Theorem gives that each ring in $G_{sar}$ decomposes as a direct sum of simple rings, i.e.\ those of the form $M_n(\mathbb{F}_q)$. Hence, we can take our primes to be
$$
  P_{sar} = \{M_n(\mathbb{F}_q)\,|\, n\geq 1, q = p^r, p \text{ prime and } r \geq 1\}, \qquad |M_n(\mathbb{F}_q)| = p^{rn^2}.
$$ 

Once again, $G_{sar}$ is not additively graded, so that the results that follow is not implies by results of \cite{warlimont1991arithmetical} and \cite{manstavivcius1992remarks}.\\

Here, write $\zeta_{G_{sar}}(s)=\sum_{n\geq1}S(n)n^{-s}$, where $S(n)$ is the number of isomorphism classes of semisimple rings of size $n$. From Theorem $3$ in \cite{knopfmacher1970arithmetical}, we have that 
$$N_{G_{sar}}(x) = \sum_{n\leq x}S(n) = \alpha_1 x + \alpha_2 x^{1/2} + O(x^{1/3}\log^2 x)$$

where $$\alpha_1 := \prod_{\substack{r,m\geq 1\\ rm^2>1}}\zeta(rm^2) = 2.498\dots \quad \text{and} \quad \alpha_2 = \zeta(\tfrac12)\prod_{\substack{r,m\geq 1\\ rm^2>2}}\zeta(\tfrac12 rm^2).$$

Taking $\delta =1, A = \alpha_1$ and $\eta = \frac{1}{2}$, then for $u = \log{x}/\log{y}$ our Theorem \ref{thm:main} gives
$$\Psi_{G_{sar}}(x,y) = (1+o(1))\,\alpha_1\,x\,\rho(u).$$
 We are not aware of any prior treatment of $\Psi_{G_{sar}}(x,y)$ in the literature and believe that this asymptotic is new.\\

In conclusion, Theorem \ref{thm:main} leads us to an explicit bound on the
number of $y$-smooth objects of bounded size in certain interesting arithmetic semigroups. For $G_{ab}$ this recovers classical results, while
for $G_{sar}$ the resulting
asymptotic for $\Psi_{G_{sar}}(x,y)$ seems to be new. It would be interesting to explore further applications of Theorem \ref{thm:main}, for example semisimple finite-dimensional algebras over $\mathbb{F}_q$, finite modules over the ring of integers of a number field, divisors of smooth algebraic varieties over number fields, nilpotent finite rings and algebras. It would also be interesting to see whether generalisations of \ref{thm:main} exist without the assumption that Axiom A holds.

%%%%%%%%%%%%%%%%%%%%%%%%%
%%%%%%%%%%%%%%%%%%%%%%%%%
%%%%%%%%%%%%%%%%%%%%%%%%%

%\bibliographystyle{alpha}
%\bibliography{references.bib}

\end{document}